\documentclass[reqno,A4paper]{amsart}

\usepackage{amsmath}
\usepackage{amssymb}
\usepackage{amsthm}
\usepackage{enumerate}
\usepackage{mathrsfs} 
\usepackage{eqlist}
\usepackage{array}

\theoremstyle{plain}
\newtheorem{theorem}{Theorem}[section]
\newtheorem{prop}[theorem]{Proposition}
\newtheorem{lemma}[theorem]{Lemma}

\theoremstyle{definition}
\newtheorem{definition}[theorem]{Definition}
\newtheorem{example}{\textit{Example}} 

\numberwithin{equation}{section}

\usepackage[mathscr]{eucal}
\usepackage{enumerate,graphicx,color}
\usepackage{url}
\usepackage{amssymb}
\usepackage{amsmath}
\usepackage{amsthm}
\usepackage{xy}
\usepackage{comment}

\usepackage{pgf,tikz}

\usetikzlibrary{calc,positioning,shapes,arrows,arrows.meta}

\tikzset{%
world/.style={circle,draw,minimum size=0.5cm,fill=gray!15},
label/.style={shape=rectangle, inner sep=6pt},
shaded/.style={draw, shape=circle, fill=black!35, inner sep=1.4pt},
unshaded/.style={draw, shape=circle, fill=white, inner sep=1.4pt},
quasi/.style={draw, shape=rectangle, rounded corners=3pt, fill=white, inner sep=2.5pt, minimum height=14.5pt},
blob/.style={draw, shape=rectangle, rounded corners=12pt, thin, densely dotted},
order/.style={thin},
curvy/.style={thin, looseness=1.2, bend angle=70},
fatcurvy/.style={thin, looseness=1.7, bend angle=75},
map/.style={->, densely dashed, shorten >=5pt, shorten <=5pt, >=stealth', looseness=1.1},
operationgj/.style={->, densely dashed, shorten >=5pt, shorten <=18pt, >=stealth', looseness=1.1},
relationlejk/.style={->, shorten >=5pt, shorten <=5pt, >=stealth'},
auto}

\newcommand{\eusbA}{\medsub e {\kern-0.75pt\A\kern-0.75pt}}

\newcommand{\bbar}[1]{{\underline{\mathbf{#1}}}}
\newcommand{\twiddle}[1]{{\smash{\underset{\raise.375ex\hbox{$\smash\sim$}}
       {\mathbf{#1}}}\vphantom{\underline{\mathbf{#1}}}}} 
\newcommand{\stwiddle}[1]{\smash{\underset{\smash{\raise.1ex\hbox{\small$\sim$}}}
                         {\mathbf{#1}}}\vphantom{#1}}
\newcommand{\twoB}{\bbar{2}}
\newcommand{\twoT}{\twiddle 2}
\newcommand{\X}{\mathbf X}
\newcommand{\mpe}[2]{\CG^{\rm mp}(#1,#2)}
\newcommand{\mph}[2]{\CL^{\rm mp}(#1,#2)}
\newcommand{\mpm}[2]{\CG_\T^{\rm mp}(#1,#2)}

\newcommand{\cat}[1]{\boldsymbol{\mathscr{#1}}}

\newcommand{\CL}{\cat L}
\newcommand{\CG}{\cat G}

\newcommand{\Lalg}{\mathbf{L}}

\font\bmi=cmmi8 scaled 1440
\newcommand{\powerset}{\raise.6ex\hbox{\bmi\char'175 }}

\newcommand{\T}{\mathscr{T}}

\DeclareMathOperator{\dom}{dom} 

\renewcommand{\le}{\leqslant}

\newcommand{\MDFIP}[2]{\langle {\uparrow} #1,{\downarrow} #2\rangle}

\newcommand{\varex}{\varepsilon_x}

\begin{document}
\title[Dual Plo\v s\v cica spaces of ortholattices]{Dual Plo\v s\v cica spaces of ortholattices}


\author[Andrew Craig  \and Miroslav Haviar]%
{Andrew Craig* \and Miroslav Haviar**}

\newcommand{\acr}{\newline\indent}

\address{\llap{*\,}Department of Mathematics and Applied Mathematics \acr
University of Johannesburg\acr
Auckland Park 2006\acr
                   SOUTH AFRICA
\acr and \acr
National Institute of Theoretical and Computational Sciences (NITheCS)\acr 
Johannesburg\acr 
 SOUTH AFRICA}
\email{acraig@uj.ac.za}

\address{\llap{**\,} Department of Mathematics\acr
Faculty of Natural Sciences\acr
M. Bel University\acr
Tajovsk\'{e}ho 40, 974 01 Bansk\'{a} Bystrica\acr SLOVAKIA
\acr and \acr
Department of Mathematics and Applied Mathematics \acr
University of Johannesburg\acr
Auckland Park 2006\acr
SOUTH AFRICA}
\email{miroslav.haviar@umb.sk}


\thanks{The first author acknowledges the National Research Foundation (NRF) of South Africa grant 127266
as well as the Slovak National Scholarship Programme.
The second author acknowledges support by Slovak VEGA grant 1/0152/22.}


\dedicatory{Dedicated to the memory of Professor G\"unther Eigenthaler.}
\begin{abstract}
We describe digraphs with topology which give dual representations of ortholattices. 
This is done via so-called dual
Plo\v{s}\v{c}ica spaces of lattices. 
First, we improve the definition of 
Plo\v{s}\v{c}ica spaces from an earlier paper to give a straight and natural generalisation of the  total order disconnectedness of Priestley spaces. Then we define the dual space of a general ortholattice as the dual Plo\v{s}\v{c}ica space of the lattice-reduct of the ortholattice equipped with a map representing the orthocomplement operation. We introduce an abstract ortho-Plo\v{s}\v{c}ica space capturing the properties of the dual space of an ortholattice, and we present dual representation theorems between general ortholattices and the ortho-Plo\v s\v cica spaces. 
We illustrate our dual representations by examples. 
\end{abstract}

\maketitle


\section{Introduction}\label{sec:intro}

In 1936 G.~Birkhoff and J.~von Neumann \cite{BN36}
proposed  taking the
lattice of closed subspaces of a Hilbert  space as an appropriate model for ``the logic of quantum mechanics''.
Such a lattice equipped with the
relation of orthogonal complement  has been described as an
ortholattice. In case the considered Hilbert space is finite-dimensional, the ortholattice of its closed subspaces is modular. Yet in case of an infinite-dimensional Hilbert space only a weaker law, the so-called orthomodular law, is satisfied as was shown by K.~Husimi \cite{H37} in 1937.    

The systematic treatment of  theory of ortholattices and orthomodular lattices was presented in the monographs by G.~Kalmbach~\cite{K83}, L.~Beran~\cite{B84},  P.~Pt\'ak and S.~Pulmannov\'a~\cite{PP91}, and by A.~Dvure\v censkij and S.~Pulmannov\'a~\cite{DP00}. These books are recommended for the reader as basic sources on orthomodular lattices and quantum structures. (For the latter see also the paper~\cite{D86} published in this journal almost forty years ago.)

In our earlier paper~\cite{P6} 
we translated a description of the Urquhart dual spaces (L-spaces) of general lattices with bounds from \cite{U78} into Plo\v{s}\v{c}ica's setting~\cite{Plos95} and we naturally called these objects \emph{Plo\v{s}\v{c}ica spaces}. These spaces are 
digraphs with topology $\mathbb{P}=(X,E,\T)$,
where $E$ is the Plo\v{s}\v{c}ica edge set and $\T$ is the  compact topology that was used by Urquhart~\cite{U78}. 
We expressed our belief in \cite{P6} that the Plo\v{s}\v{c}ica spaces are easier to work with than the Urquhart L-spaces equipped with two quasi-order relations. Also we believe that recovering the lattice from its Plo\v{s}\v{c}ica dual via the maximal partial morphisms is easier and more natural than  recovering the lattice from its dual L-space via the doubly closed stable sets in the Urquhart representation. Another advantage of using the Plo\v{s}\v{c}ica spaces is that in case the lattice $\Lalg$ is distributive, the edge relation $E$ in a Plo\v{s}\v{c}ica space becomes the order relation used in 
Priestley duality and the Plo\v{s}\v{c}ica space becomes a Priestley space.

In Section~\ref{sec:Plospaces} of this paper we show that our definition of the Plo\v{s}\v{c}ica spaces from \cite{P6} can be simplified such that it gives a straight and natural generalisation of the  total order disconnectedness of Priestley spaces (cf. \cite{Pr70}, \cite{Pr72}). After simplifying Plo\v{s}\v{c}ica spaces we will mention other results from~\cite{P6} concerning 
Plo\v{s}\v{c}ica spaces that will be needed in our dual representation theorems for ortholattices.

In Section~\ref{sec:orthospaces} we define the dual space of a general ortholattice 
$\Lalg'$
as the dual Plo\v{s}\v{c}ica space of the lattice-reduct of the ortholattice 
$\Lalg'$ equipped with a map $g$, which 
will represent the orthocomplement operation on the dual side. We present the axioms 
satisfied
by this map $g$ on the dual 
space
and then we define an abstract ortho-Plo\v{s}\v{c}ica space as a Plo\v{s}\v{c}ica space equipped with such a $g$-map satisfying the 
axioms.  

In Section~\ref{sec:dual repr} we present the dual representation theorems between general ortholattices and the ortho-Plo\v{s}\v{c}ica spaces. These are based on the dual representation theorems between general lattices with bounds and 
Plo\v{s}\v{c}ica spaces presented in~\cite{P6}. 

In Section~\ref{sec:examples} we illustrate our 
results on three examples.  

\section{Preliminaries}\label{sec:prelim} 

An \emph{ortholattice\/} is an algebra 
$\Lalg' 
= (L;\lor ,\land ,',0,1)$, where $(L;\lor,\land,0,1 )$ is a bounded lattice and~$'$ is the unary operation of 
orthocomplementation, i.e. the 
following identities are satisfied for all $a, b\in L$:
\begin{itemize}
\item[] $(a')' = a$, $a \land a'= 0$, $a \lor a' = 1$, 
$(a \land b)' = a'\lor b'$, $(a \lor b)' = a'\land b'$.
\end{itemize}

By an \emph{orthomodular lattice\/} is meant an ortholattice 
satisfying the \emph{orthomodular law}, which says that for all $a, b\in L$, 
$$
a \leqslant 
b\ \Rightarrow \ b = a \lor (b\land a').
$$
An equivalent form of the orthomodular law is the identity ($a, b\in L$)
$$b = (b\land a)\lor [b \land (b \land a)'].
$$

To compare 
orthomodular lattices with 
Boolean algebras, various types of symmetric differences can be defined 
on orthomodular lattices (see~\cite{DDL96} and  \cite{D05}), which then coincide with the conventional symmetric difference 
in the case of Boolean algebras. 
Also, the distributive law can be used when working with Boolean algebras, but does not  in general hold in orthomodular lattices. 

Now we recall basic concepts
and facts  from~\cite{P6} that will be needed later.

By a \emph{partial homomorphism} from a lattice 
$\Lalg = (L;\vee,\wedge,0,1)$ 
into the 
lattice
$\twoB=(\{0,1\};\vee,\wedge,
0,1)$ is meant a partial map $f : L \to \{0,1\}$ such that $\dom f$ is a $(0,1)$-sublattice of $\Lalg$ and 
$f\!\! \upharpoonright_{\dom f}$ is a $(0,1)$-lattice homomorphism. By a \emph{maximal partial homomorphism}
(briefly \emph{MPH}) is meant a
partial homomorphism with no proper extension. We will denote by $\mph{\Lalg}{\twoB}$ the set of all MPHs from $\Lalg$ into~$\twoB$.

Plo\v{s}\v{c}ica's binary relation on the set
$P_{\Lalg}= \mph{\Lalg}{\twoB}$ 
of MPHs from $\Lalg$ to $\twoB$  
is defined  (see~\cite{Plos95} or \cite[p.~4]{P6}) such that for any  MPHs  $f, g \in P_{\Lalg}$,
\begin{itemize}
\item[(E1)] $(f,g)\in E \quad
\iff \quad (\forall x \in \dom f \cap \dom g) \ f(x) \le g(x)$.
\end{itemize}

For a general lattice $\Lalg = 
(L;\vee,\wedge,0,1)$ 
with bounds, the MPHs from $\Lalg$ to $\twoB$ naturally correspond to the \emph{maximal disjoint filter-ideal pairs} (briefly \emph{MDFIPs}) of $\Lalg$ as we recall now their definition.  (See e.g. \cite[p.~76]{Plos95} or \cite[p.~4]{P6}.)

\begin{definition}\label{def:DFIP}
A pair $\langle F, I \rangle$ is said to be a disjoint filter-ideal pair (briefly DFIP) of a lattice $\Lalg$ with bounds if $F$ is a filter of $\Lalg$ and $I$ is an ideal of $\Lalg$ such that $F \cap I = \varnothing$. By a maximal disjoint filter-ideal pair $\langle F, I \rangle$ is meant a DFIP of $\Lalg$ such that there is no DFIP $\langle G, J\rangle \neq \langle F, I \rangle$ with $F \subseteq G$ and $I \subseteq J$.
\end{definition}

When the base set of our dual space to $\Lalg$ will
be the set $X_{\Lalg}$ of all MDFIPs of $\Lalg$, 
then the Plo\v{s}\v{c}ica relation $E$ is determined on 
$X_{\Lalg}$ such that (see \cite[p.~76]{Plos95} or~\cite[p.~5]{P6}) for any two MDFIPs $\langle F, I \rangle$ and $\langle G, J \rangle$, 
\begin{itemize}
\item[(E2)] $\langle F, I \rangle E \langle G, J \rangle \quad
\iff \quad F \cap J = \emptyset$.
\end{itemize}

In Craig, Gouveia and Haviar~\cite{P3} we described properties of the digraphs
$(V,E)$ 
dual to general lattices with bounds. We identified three properties (Ti), (R), (S) of these digraphs, which 
were called \emph{TiRS graphs} in~\cite{P3} (see also~\cite{P4}). Yet starting from our more recent papers~\cite{P5} and~\cite{P7}, we prefer to use the terminology 
\emph{TiRS digraphs}. To present their definition, we recall that we use the notations $xE = \{\, y \in V \mid (x,y) \in E \,\}$ and $Ex = \{\, y \in V \mid (y,x) \in E  \,\}$.

\begin{definition}[{\cite[Definition 2.2]{P3}}]\label{def:TiRS} A TiRS digraph $G = (V,E)$ is a set $V$ and a
reflexive relation $E \subseteq V\times V$ such that:
\begin{itemize}
\item[(S)] If $x,y \in V$ and $x\neq y$ then $xE\neq yE$ or $Ex \neq Ey$.
\item[(R)] For all $x,y \in V$, if $xE \subset yE$ then $(x,y)\notin E$, and
if $Ex \subset Ey$ then $(y,x)\notin E$.
\item[(Ti)] For all $x,y \in V$, if
$(x,y)\in E$ then there exists $z \in V$ such that $zE \subseteq xE$ and $Ez \subseteq Ey$.
\end{itemize}
\end{definition}

It is important to recall that by \cite[Proposition 2.3]{P3} 
(see also~\cite[p.~5]{P6}),
the dual digraphs
$(P_\Lalg,E)$ and $(X_\Lalg,E)$ of any general lattice $\Lalg$ with bounds (possibly infinite)
are TiRS digraphs.

We also notice that (S) in 
the 
definition above stands for ``separation''.
If $G=(P_\Lalg,E)$, then for $x,y \in P_\Lalg$, $x\ne y$ obviously implies $x^{-1}(1)\ne y^{-1}(1)$ or $x^{-1}(0)\ne y^{-1}(0)$. Similarly, the condition (R) is an abstract version of the property of the dual $G=(P_\Lalg,E)$, which says that if $x^{-1}(1)\supset y^{-1}(1)$, then as 
$\langle y^{-1}(1),y^{-1}(0) \rangle$
is an MDFIP, we cannot have $x^{-1}(1) \cap  y^{-1}(0)=\emptyset$, i.e.
we cannot have 
 $(x,y)\in E$.  
The intuition behind the (Ti) condition is that when the lattice $\Lalg$ is distributive and the relation $E$ becomes the Priestley partial ordering, then (Ti) says that whenever $x\le y$, then there exists $z$ in the interval $[x,y]$. 

In~\cite{P3} we extended the classical Birkhoff dual representation between finite distributive lattices and finite posets from the 1930s into a dual representation between all finite lattices and all finite TiRS digraphs. 

In \cite{P6} we proved that  
TiRS digraphs satisfy the following ``weak antisymmetry'' condition (wA):

\begin{prop}[{\cite[Lemma 2.3]{P6}}]\label{lem:antisym}
Let $G=(V,E)$ be a TiRS digraph. Then 
\begin{itemize}
\item[(wA)] For all $x,y \in V$, if $xE \subseteq yE$ and $Ex \subseteq Ey$, then $x=y$.
\end{itemize}
\end{prop}

Now we recall facts concerning general
digraphs $G =(X,E)$ that will be needed to present the Plo\v{s}\v{c}ica dual representation theorem from~\cite{Plos95}. Let $\twoT=(\{0,1\}, \le)$ be the two-element digraph with the relation $E$ being the partial order $\le$ on $\{0,1\}$. A~partial map $\varphi \colon X \to \twoT$ is said to preserve the relation $E$ if for $x, y \in \dom \varphi$, $\varphi(x)\leqslant \varphi(y)$ 
whenever 
$(x,y)\in E$. By $\mpe{G}{\twoT}$ we denote the 
set of all 
maximal partial $E$-preserving maps from $G$ to $\twoT$.

From the lemma below we have that for a digraph $G=(X,E)$ and $\varphi, \psi \in \mpe{G}{\twoT}$,
$$
\varphi^{-1}(1) \subseteq \psi^{-1}(1) \:\Longleftrightarrow\: \psi^{-1}(0) \subseteq \varphi^{-1}(0).
$$

\begin{lemma}
[cf.~{\cite[Lemma 1.3]{Plos95}}]
\label{lem:Plos1.3}
Let $G=(X,E)$ be a digraph and let us consider
$\varphi \in \mpe{G}{\twoT}$. Then
\begin{enumerate}[{\upshape (i)}]
\item $\varphi^{-1}(0)=\{\, x\in X \mid \text{there is no}\ y\in\varphi^{-1}
(1)\ \text{with}\ (y,x)\in E \,\}$;
\item $\varphi^{-1}(1)=\{\, x\in X \mid
\text{there is no}\ y\in\varphi^{-1}
(0)\ \text{with}\ (x,y)\in E\,\}$.
\end{enumerate}
\end{lemma}

Hence, 
the reflexive and transitive binary relation $\leqslant$ defined on
$\mpe{G}{\twoT}$ by $\varphi\le\psi$ if and only if $\varphi^{-1}(1)\subseteq\psi^{-1}(1)$ is a partial order.

We remark that Plo\v{s}\v{c}ica's dual of a lattice $\Lalg$ was in~\cite{Plos95} given as the digraph with topology 
$\mathcal{D}({\Lalg})= (P_{\Lalg}, E, \T_{\Lalg})$ having as a base set the set 
all MPHs from $\Lalg$ into $\twoB$, where the topology $\T_{\Lalg}$ has as a subbasis of closed sets all sets of the form
\[
V_a=\{\,f \in \mph{\Lalg}{\twoB} \mid f(a)=0 \,\} \quad \text{and} \quad
W_a=\{\,f \in \mph{\Lalg}{\twoB} \mid f(a)=1 \,\}
\]
for all $a\in L$. This topology $\T_{\Lalg}$ is $T_1$ and is compact, and it turns out to be the same topology as used by Urquhart (\cite[Lemma 6]{U78}).

We will also need to recall facts concerning general digraphs $\mathbb{P}= (X,E,\T)$ with topology. (See also~\cite{P1}.) A map $\varphi \colon (X_1,E_1,\T_1)\to (X_2,E_2,\T_2)$ is 
a \emph{morphism} if it preserves the relation and is continuous as a map from  $(X_1,\T_1)$ to  $(X_2,\T_2)$.
A \emph{partial morphism} is a partial map $\varphi\colon  (X_1,E_1,\T_1)\to (X_2,E_2,\T_2)$
whose domain is a $\T_1$-closed subset of $X_1$ and the
restriction of $\varphi$ to its domain is a morphism.
Then by a \emph{maximal partial morphism} (MPM) is meant a partial morphism such that there is no partial morphism properly extending it. By $\mpm{\mathbb{P}}{\twoT_\T}$ we denote the set of MPMs from a digraph with topology  $\mathbb{P}= (X,E,\T)$  to the two-element
digraph $\twoT_\T$ with the discrete topology. 

Now we present the Plo\v s\v cica representation theorem for general lattices with bounds:

\begin{prop}[{\cite[Lemmas 1.2, 1.5, Theorem~1.7]{Plos95}, \cite[Proposition 2.5]{P6}}]\label{lem-eval}

Let\, $\Lalg$ be a  lattice with bounds and let\, $\mathcal{D}(\Lalg) = (P_{\Lalg}, E,
\T_{\Lalg})$  
be the dual
of\, $\Lalg$.  For $a\in L$, let the evaluation
map\, $e_a \colon \mathcal{D}(\Lalg) \to
\twoT_\T$
be defined by
\[
e_a(f)= \begin{cases}
 f(a) &\text{ $a\in \dom(f)$,}\\
  - &\text{ undefined otherwise.}
    \end{cases}
       \]
Then the following hold:
\begin{enumerate}[{\upshape (i)}]
\item The map\, $e_a$ is an element of
$\mpm{\mathcal{D}(\Lalg)}{\twoT_\T}$
for each  $a \in L$.
\item Every\,
$\varphi \in
\mpm{\mathcal{D}(\Lalg)}{\twoT_\T}$
is of the form $e_a$ for some $a\in L$.
\item The map\, $e_{\Lalg}: \Lalg \to
\mpm{\mathcal{D}(\Lalg)}{\twoT_\T}$
given by evaluation, $a \mapsto e_a$
{\upshape(}$a \in L${\upshape)}, is an
isomorphism of\, $\Lalg$ onto the lattice\,
$\mpm{\mathcal{D}(\Lalg)}{\twoT_\T}$,
ordered by the relation\, $\varphi\le\psi$ if and only if\, $\varphi^{-1}(1)\subseteq\psi^{-1}(1)$.
\end{enumerate}
\end{prop}

The next two lemmas will be used in Section~\ref{sec:orthospaces}.

\begin{lemma}\label{lem:dual}
The dual $\mathcal{D}(\Lalg) = (P_{\Lalg},E,\T_{\Lalg})$ of  a lattice $\Lalg$ with bounds satisfies the following conditions for any $x,y \in P_{\Lalg}$:
\begin{itemize}
    \item[(i)] $xE \subseteq yE$ if and only if $y^{-1}(1) \subseteq x^{-1}(1)$;
    \item[(ii)] $Ex \subseteq Ey$ if and only if $y^{-1}(0) \subseteq x^{-1}(0)$.
\end{itemize}
\end{lemma}

\begin{proof}
We prove (i) while (ii) can be shown analogously. If $xE \subseteq yE$ and $y(a)=1$ for $a\in L$, then using reflexivity of $E$ we get $x\in yE$, whence $(y,x)\in E$, thus $x(a)=1$. Conversely, if $y^{-1}(1) \subseteq x^{-1}(1)$ and $z\in xE$ for $z\in P_{\Lalg}$, then $(x,z)\in E$, i.e. $x^{-1}(1) \cap z^{-1}(0) =\emptyset$. Hence $y^{-1}(1) \cap z^{-1}(1) =\emptyset$, thus $(y,z)\in E$, i.e. $z\in yE$ as required.      
\end{proof}

The lemma below taken from \cite{P6} is usually applied to situations where $S$ is a filter or an ideal. 

\begin{lemma}[{\cite[Lemma 4.5]{P6}}]\label{lem:S}
Let $\Lalg$ be a lattice and $S\subseteq L$. 

\begin{itemize}
\item[(i)] If $I$ is an ideal 
of $\Lalg$ maximal w.r.t. being disjoint from the set $S$, then for every element $y\notin I$ there exists an element $x\in I$ such that $x\vee y \in S$.
\item[(ii)] If $F$ is a filter of 
$\Lalg$ maximal w.r.t. being disjoint from the set $S$, then for every element $y\notin F$ there exists an element $x\in F$ such that $x\wedge y \in S$.
\end{itemize}
\end{lemma}

We finally recall some facts from~\cite{P2} (see also \cite{GW99} and \cite{ACthesis}).  
Let $G=(X,E)$ be a digraph. By a \emph{context} associated to the complement $E^\complement$ of the digraph relation $E$ is meant the triple $\mathbb{K}(G) := (X,X,E^\complement)$, where $E^\complement$ is given by $E^\complement = (X\times X){\setminus} E$. Then a Galois connection is defined via so-called \emph{polars} such that the maps
$$E^\complement_\triangleright : (\powerset(X),\subseteq) \to (\powerset(X),\supseteq)\quad \text{and}\quad E^\complement_\triangleleft : (\powerset(X),\supseteq) \to (\powerset(X),\subseteq)$$
are given by
\begin{align}
E^\complement_\triangleright (Y) &=\{\, x \in X \mid (\forall\, y \in Y) (y,x) \notin E \,\},\notag\\
E^\complement_\triangleleft (Y) &=\{\, z \in X \mid (\forall\, y \in Y) (z,y) \notin E \,\}.\notag
\end{align}
Then the 
\emph{concept lattice} $\mathrm{CL}(\mathbb{K}(G))$  of the
context $\mathbb{K}(G)= (X,X,E^\complement)$ is a complete lattice, ordered by inclusion, given by
$ \mathrm{CL}(\mathbb{K}(G)) = \{\, Y \subseteq X \mid
(E^\complement_\triangleleft \circ E^\complement_\triangleright) (Y)= Y \,\}$.

\section{Simplifying Plo\v{s}\v{c}ica spaces}\label{sec:Plospaces}

At the beginning of this section we will show that our definition of 
Plo\v{s}\v{c}ica spaces from \cite{P6} can be simplified such that it gives a straight and natural generalisation of the  total order disconnectedness of Priestley spaces (cf. \cite{Pr70}, \cite{Pr72}).

We start with recalling our definition of the Plo\v{s}\v{c}ica spaces from \cite{P6}.

\begin{definition}[{\cite[Definition 3.1]{P6}}]\label{def:ploscicaspaces}
A \emph{Plo\v{s}\v{c}ica space} is a structure
$\mathbb{P}=(X,E,\T)$
such that
\begin{enumerate}
 \item
 $(X,E)$ is a TiRS digraph and $(X,\T)$ is a compact topological space.
 \item $\mathbb{P}$ is doubly-disconnected, i.e. for any $x, y \in X$:
  \begin{enumerate}[(a)]
 \item If $yE \not\subseteq xE $ then there exists $\varphi \in
  \mpm{\mathbb{P}}{\twoT_\T}$  such that $\varphi(x)=1$ and $\varphi(y)\neq 1$.
 \item If $Ey \not\subseteq Ex$, then there exists $\varphi\in
\mpm{\mathbb{P}}{\twoT_\T}$
 such that $\varphi(x)=0$ and $\varphi(y) \neq 0$.
 \end{enumerate}
\item For any $\varphi$, $\psi \in\mpm{\mathbb{P}}{\twoT_\T}$,
the sets
$$ E_{\triangleleft}^{\complement}(\varphi^{-1}(0) \cap \psi^{-1}(0)) \quad \text{ and } \quad E_{\triangleright}^{\complement}(\varphi^{-1}(1)\cap \psi^{-1}(1))$$
are closed.
 \item The family
 $$ \left\{\, X{\setminus}\varphi^{-1}(1)\mid \varphi \in
\mpm{\mathbb{P}}{\twoT_\T}
\,\right\} \cup \left\{\, X{\setminus}\varphi^{-1}(0) \mid \varphi\in
\mpm{\mathbb{P}}{\twoT_\T}
\,\right\} $$
 forms a subbase for $\T$.
\end{enumerate}
\end{definition}

Let us now recall that the  total order disconnectedness of Priestley spaces 
means that for any two points $x \ne y$ there exists a clopen up-set
(or, equivalently, down-set) that separates them. In \cite{P6}
we interpreted the condition of the 
double disconnectedness 
(2) as a generalisation of the total order disconnectedness of Priestley spaces in the following sense: As from  Lemma~\ref{lem:antisym} we have in TiRS digraphs that  $x\ne y$ yields $yE \nsubseteq xE$ or $Ey\nsubseteq Ex$, the 
double disconnectedness (2) can be thought of as saying that $x \ne y$ implies there is an MPM
$\varphi \in \mpm{\mathbb{P}}{\twoT_\T}$
for which $\varphi(x)=1$ and $\varphi(y)\neq 1$ or
there is an MPM $\psi \in \mpm{\mathbb{P}}{\twoT_\T}$
for which $\psi(x)=0$ and $\psi(y) \neq 0$.

To simplify the condition (2) of the 
double disconnectedness
in Definition~\ref{def:ploscicaspaces} above and to present it as a more natural generalisation of the total order disconnectedness of Priestley spaces, we 
recall the following result 
from~\cite{P6}. We now refer to the condition in the result as the \emph{(TED) condition}, standing for \emph{total edge-disconnectedness}.

\begin{prop}[{\cite[Corollary 3.11]{P6}}]\label{P6:3.11}
Let $\mathbb{P}=(X,E,\T)$ be a Plo\v{s}\v{c}ica space. The following holds:

\noindent 
{\upshape (TED)} \ For any $x,y\in X$ such that $(x,y)\notin E$ there exists $\varphi\in\mpm{\mathbb{P}}{\twoT_\T}$ such that $\varphi(x)=1$ and $\varphi(y)=0$.
\end{prop}

We will prove that the new condition (TED) is equivalent to the 
somewhat
more complicated condition (2) 
in Definition~\ref{def:ploscicaspaces}. It is important to notice that the simplified condition of 
total edge-disconnectedness (TED) is exactly the total order disconnectedness if one replaces the Plo\v{s}\v{c}ica digraph relation $E$ with the Priestley order relation 
$\leqslant$ 
and an MPM with an order-preserving continuous function.

In the following two results the conditions (1)--(4) refer to Definition~\ref{def:ploscicaspaces}.

\begin{lemma}\label{impl:TED}
Let $\mathbb{P}=(X,E,\T)$ be a digraph with topology.  
Then the condition 
{\upshape (TED)}
implies 
condition 
{\upshape (2)}. 
\end{lemma}

\begin{proof} If $yE \nsubseteq xE$ then there exists $z$ such that $(y,z) \in E$ but $(x,z) \notin E$. Since $(x,z) \notin E$, by (TED) there exists an MPM $\varphi$ such that $\varphi(x)=1$  and $\varphi(z)=0$.  Since $(y,z)\in E$ and $ \varphi(z)=0$, we have that $\varphi(y)\neq 1$. This gives us 2(a).

To show 2(b), if $Ey \nsubseteq Ex$ then there exists $z$ such that $(z,y) \in E$ and $(z,x) \notin E$. By (TED) there exists an MPM $\varphi$ such that $\varphi(z)=1$ and $\varphi(x)=0$. Since $(z,y) \in E$ and $\varphi(z)=1$, we get that $\varphi(y)\neq 0$. 
\end{proof}

The next result immediately follows from the previous lemma and Proposition~\ref{P6:3.11}: 

\begin{prop}\label{equiv:TED}
Let $\mathbb{P}=(X,E,\T)$ be a digraph with topology satisfying the conditions {\upshape (1)}, {\upshape (3)} and {\upshape (4)}. Then 
condition {\upshape (2)} 
is equivalent to the condition {\upshape (TED)}.
\end{prop}

Let us recall that the Plo\v{s}\v{c}ica dual of a lattice $\Lalg$ with bounds is the digraph with topology 
$\mathcal{D}({\Lalg})= (P_{\Lalg}, E, \T_{\Lalg})$. 

In \cite[Lemma~3.2]{P6} we showed the 
double disconnectedness
of this structure by proving the condition (2) of Definition~\ref{def:ploscicaspaces}. Now with the new simplified condition for the total edge-disconnectedness,  
the proof can be seen as 
straightforward: 

\begin{prop}\label{proof:TED}
Let $\Lalg$ be a lattice with bounds and $\mathcal{D}({\Lalg})= (P_{\Lalg}, E, \T_{\Lalg})$ be its Plo\v{s}\v{c}ica dual. Then $\mathcal{D}({\Lalg})$ is totally edge-disconnected. 
\end{prop}

\begin{proof}
If $f, g\in P_{\Lalg}$ and $(f,g)\notin E$, then by (E1) 
in Section~\ref{sec:prelim} there is  $a \in \dom f \cap \dom g$ with $f(a)=1$ and $g(a)=0$. By Proposition~\ref{lem-eval} for $\varphi =e_a$ as an MPM from the set $\mpm{\mathcal{D}(\Lalg)}{\twoT_\T}$ we immediately get 
$\varphi(f) =e_a(f)= f(a) = 1$ and $\varphi(g) =e_a(g)= g(a) = 0$ showing the condition (TED).    
\end{proof}

We now 
mention other results from~\cite{P6} concerning 
Plo\v{s}\v{c}ica spaces that will be needed in our dual representation theorems for 
ortholattices in Section~\ref{sec:dual repr}.

For a Plo\v{s}\v{c}ica space $\mathbb{P}=(X,E,\T)$ we define the ordering $\leqslant$ on $\mpm{\mathbb{P}}{\twoT_\T}$ by
$$\varphi \leqslant \psi \quad \Longleftrightarrow \quad
\varphi^{-1}(1) \subseteq \psi^{-1}(1)
\quad \Longleftrightarrow \quad
\psi^{-1}(0) \subseteq \varphi^{-1}(0).
$$
We proved in~\cite[Proposition~3.7]{P6} that $\mathcal{L}(\mathbb{P})=(\mpm{\mathbb{P}}{\twoT_\T},\leqslant)$ is a bounded lattice, which we considered to be the \emph{dual lattice} 
of the Plo\v{s}\v{c}ica space $\mathbb{P}=(X,E,\T)$.
Moreover, in \cite[Theorem 3.8]{P6} we used the map $\nu: \Lalg \to
\mathcal{L}(\mathcal{D}(\Lalg))$ given by $\nu (a)= e_a$ to present a short proof of the Plo\v s\v cica representation theorem for general lattices with bounds, which is given in~(iii) of Proposition~\ref{lem-eval}. 

To obtain our dual representation theorem for the Plo\v{s}\v{c}ica spaces, we said in~\cite{P6} that  two Plo\v{s}\v{c}ica spaces $\mathbb{P}_1$ and $\mathbb{P}_2$ are \emph{digraph-homeomorphic} and used the notation $\mathbb{P}_1 \cong \mathbb{P}_2$, if there exists a map 
$\vartheta: X_1 \to X_2$ such that $xE_1y$ iff $\vartheta(x)E_2 \vartheta(y)$ and $\vartheta$ is a (topological) homeomorphism. Such a digraph-homeomorphism $\vartheta$ from $\mathbb{P}=(X,E,\T)$ to
$\mathcal{D}(\mathcal{L}(\mathbb{P}))$ is assigning to $x \in X$
the partial map $\vartheta(x) = \varepsilon_x: \mathcal{L}(\mathbb{P}) \to \twoB$, which is given as follows:

\begin{lemma}[{\cite[Lemma 3.9]{P6}}]\label{lem:defn-v}
Let $\mathbb{P}=(X,E,\T)$ be a Plo\v{s}\v{c}ica space. For $x \in X$, define a partial map
$\varepsilon_x$ from $\mathcal{L}(\mathbb{P})$ to
$\twoB$
such that
for $\varphi \in \mpm{\mathbb{P}}{\twoT_\T}$
$$ \varex(\varphi) = \begin{cases}
\varphi(x) & \text{ if } x \in\dom\varphi,
\\
- & \text{otherwise.}
\end{cases}
$$
\noindent
Then $\varex \in \mph{\mathcal{L}(\mathbb{P})}{\twoB}$.
\end{lemma}

Finally, we used the above map $\vartheta$ to
show that every Plo\v{s}\v{c}ica space is digraph-homeomorphic to its second dual:

\begin{theorem}[{\cite[Theorem 3.10]{P6}}]\label{thm:graph-homeomorphic}
Let $\mathbb{P}=(X,E,\T)$ be a Plo\v{s}\v{c}ica space. Then
$\mathbb{P} \cong \mathcal{D}
(\mathcal{L}(\mathbb{P}))$.
\end{theorem}

The following lemma concerning 
Plo\v{s}\v{c}ica spaces will be used in Section~\ref{sec:dual repr}
(see also \cite[Lemma~2.6]{P6}): 

\begin{lemma}\label{lem:E}
Let 
$\mathbb{P} 
=(X,E,\tau)$ be a Plo\v{s}\v{c}ica space and 
$\varphi \in \mpm{\mathbb{P}}{\twoT_\T}$. 
Let $x,y \in X$.

\begin{itemize}
    \item[(i)] If $xE \subseteq yE$ and $\varphi(y)=1$, then $\varphi(x)=1$.
    \item[(ii)] If $Ex \subseteq Ey$ and $\varphi(y)=0$, then $\varphi(x)=0$.
\end{itemize}
\end{lemma}

\begin{proof}
We prove (i), while (ii) can be shown analogously. Let $xE \subseteq yE$ and $\varphi(y)=1$. Suppose that $\varphi(x)\ne1$. As $x\notin \varphi^{-1}(1)$, by Lemma~\ref{lem:Plos1.3} there is $z\in X$ such that $(x,z)\in E$ and $\varphi(z)=0$. Since $z\in xE$ and $xE \subseteq yE$, we have $(y,z)\in E$. As $\varphi(y)=1$, it follows $\varphi(z)\ne 0$, a contradiction.
\end{proof}

\section{Ortho-Plo\v{s}\v{c}ica spaces}\label{sec:orthospaces}

After simplifying the definition of Plo\v{s}\v{c}ica spaces in the previous section and 
recalling the dual 
representations 
between 
general lattices with bounds and 
Plo\v{s}\v{c}ica spaces, we are now ready to establish, in the same spirit, a dual space for a general ortholattice, which will be called an \emph{ortho-Plo\v{s}\v{c}ica space}.

So to switch again to ortholattices, we firstly present 
the following lemma that will be used in the proof of Theorem~\ref{prop:props-of-gmap}.

\begin{lemma}\label{lem:FIo}
Let 
$\Lalg'$ 
be an ortholattice and  
$f \in P_{\Lalg}$.
Let us denote by 
$\langle F,I \rangle$ 
the corresponding MDFIP 
$\langle f^{-1}(1),f^{-1}(0) \rangle$.
We define sets 
\begin{align}
F' &:= \{a'\in L \mid a\in F\}=\{a\in L \mid a'\in F\},\notag\\ 
I' &:= \{b'\in L \mid b\in I\}=\{b\in L \mid b'\in I\}.\notag
\end{align}
Then 
$\langle I',F' \rangle$ 
is an MDFIP.
\end{lemma}

\begin{proof}
To show that $I'$ is a filter 
of $\Lalg$, let us assume $a\in I'$, $b\in L$ and $a
\leqslant b$. Then $a'\in I$ and $b'
\leqslant 
a'$, whence $b'\in I$. Now let $c,d\in I'$, i.e. $c',d'\in I$.  
Since $I$ is an ideal, $c'\vee d' \in I$, which gives $(c'\vee d')'=c \wedge d \in I'$, hence $I'$ is a filter. 
Dually, $F'$ is an ideal 
of $\Lalg$.

As $\langle F,I \rangle$ 
is a DFIP, 
$\langle I',F' \rangle$ is also a DFIP,
since $a\in F\cap I$ iff $a'\in F'\cap I'$. 

Without loss of generality, 
suppose that the filter $I'$ can be properly extended to a filter $\nabla$ disjoint from the ideal $F'$, i.e. there exists $e\in \nabla$ such that $e\notin 
I' \cup F'$. 
Then it is easy to show that $\nabla':= \{a'\in L \mid a\in \nabla\}=\{a\in L \mid a'\in \nabla\}$ is an ideal properly extending the ideal $I$ and disjoint from the filter $F$, a contradiction. (We observe 
that $e'\in \nabla'$ and $e'\notin 
I\cup F$.) 
\end{proof}

To define an ortho-Plo\v{s}\v{c}ica space as a dual space for a general ortholattice 
$\Lalg'$
we will need to find a map on the dual Plo\v{s}\v{c}ica space of the lattice-reduct of the ortholattice 
$\Lalg'$,
which would on the dual side represent the orthocomplement operation. 
We remark that the use of our concept of an MPH below always refers to the lattice-reduct of an ortholattice 
$\Lalg'$. 

\begin{lemma}\label{lem:gmap} 
Let 
$\Lalg'=
(L;\vee, \wedge, ',0,1)$ be an ortholattice.
For an MPH $f:\Lalg \to \mathbf{2}$ and $a \in L$, define $(g(f))(a)=0$ if $f(a')=1$ and $(g(f))(a)=1$ if $f(a')=0$.
Then $g(f) \in \mph{\Lalg}{\twoB}$,  
i.e. is an MPH. 
\end{lemma}
\begin{proof}
We prove that 
$\langle g(f)^{-1}(1),g(f)^{-1}(0) \rangle$
is an MDFIP. 

To show that $g(f)^{-1}(1)$ is a filter in $\Lalg$, let us firstly assume $a\in g(f)^{-1}(1)$, $b\in L$ and $a\leq b$. Then $g(f)(a)=1$, i.e. $f(a')=0$ and $b'\leq a'$. Since $f$ as a (partial) lattice homomorphism is order-preserving, we obtain $f(b')=0$, thus $b\in g(f)^{-1}(1)$. Secondly, let $c,d\in g(f)^{-1}(1)$, i.e. $f(c')=0$ and $f(d')=0$. 
Then $f((c\wedge d)')=f(c'\vee d')=f(c')\vee f(d')=0$, whence $c \wedge d \in g(f)^{-1}(1)$  
showing that $g(f)^{-1}(1)$ is a filter in $\Lalg$.

Dually, one can prove that  $g(f)^{-1}(0)$ is an ideal in $\Lalg$.

To show that 
$\langle g(f)^{-1}(1),g(f)^{-1}(0) \rangle$ 
is a DFIP, suppose that $a\in g(f)^{-1}(1) \cap g(f)^{-1}(0)$ for some $a\in L$. Then $a'\in f^{-1}(0) \cap f^{-1}(1)$, which contradicts that $f$ is a (partial) lattice homomorphism.

Finally, to show the maximality of the DFIP 
$\langle g(f)^{-1}(1),g(f)^{-1}(0) \rangle$
suppose that $g(f)^{-1}(1)$ can be extended to some filter $F$ disjoint from $g(f)^{-1}(0)$ and that $F$ includes 
an element $b\notin g(f)^{-1}(1)$. Then we have $b'\notin f^{-1}(0)$. By Lemma~\ref{lem:S} with $S=f^{-1}(1)$ we have that there is $c\in f^{-1}(0)$ such that $b'\vee c \in f^{-1}(1)$, whence $b\wedge c' \in g(f)^{-1}(0)$. Since $c'\in g(f)^{-1}(1)$, we get $c'\in F$, thus $b\wedge c' \in F$, which contradicts that $F$ is disjoint from $g(f)^{-1}(0)$.  
\end{proof}

It is important to show that the map $g$ is compatible with the structure of the dual Plo\v{s}\v{c}ica space of the lattice-reduct of an ortholattice.

\begin{prop}\label{prop:gmap-cont}
Let $\Lalg'$ be an ortholattice and $\mathcal{D}({\Lalg})= (P_{\Lalg}, E, \T_{\Lalg})$ be the dual Plo\v{s}\v{c}ica space of the lattice-reduct $\Lalg$ of $\Lalg'$. Then the map $g: P_{\Lalg} \to P_{\Lalg}$
is continuous.     
\end{prop}

\begin{proof}
We recall that our topology on $P_{\Lalg}$ has as 
its
subbase of open sets the sets $P_{\Lalg}\setminus \varphi^{-1}(0)$ and $P_{\Lalg}\setminus \varphi^{-1}(1)$ for all $\varphi\in \mpm{\mathcal{D}({\Lalg})}{\twoB}$. Because $\mathcal{D}({\Lalg})$ is a dual space of a lattice $\Lalg$, by Plo\v{s}\v{c}ica's representation theorem for $\Lalg$ we can assume that every $\varphi\in \mpm{\mathcal{D}({\Lalg})}{\twoB}$ is an evaluation map $e_a$ for some $a\in L$. Hence we may assume that our topology on $\mathcal{D}({\Lalg})$ has as  its 
subbase 
of open sets the sets  $P_{\Lalg}\setminus 
V_a$ 
and $P_{\Lalg}\setminus 
W_a$  
for all $a\in L$, where 
$V_a 
=\{x\in P_{\Lalg}\mid e_a(x)= x(a)=0\}$ and $W_a 
=\{x\in P_{\Lalg}\mid e_a(x) = x(a)=1\}$. (See also Plo\v{s}\v{c}ica's paper~\cite{Plos95}.) 

Now we are going to show that the pre-images of subbasic open sets in $g$ are again open. We have that for any $a\in L$,
$$
g^{-1}(P_{\Lalg}\setminus 
W_a) 
= \{x\in P_{\Lalg}\mid g(x)(a)\ne 1\}=\{x\in P_{\Lalg}\mid x(a')\ne 0\}= P_{\Lalg}\setminus 
V_{a'},  
$$
which is an open set. Analogously, 
$$
g^{-1}(P_{\Lalg}\setminus 
V_a)  
= \{x\in P_{\Lalg}\mid g(x)(a)\ne 0\}=\{x\in P_{\Lalg}\mid x(a')\ne 1\}= P_{\Lalg}\setminus 
W_{a'} 
$$
is an open set. Hence $g$ is continuous.
\end{proof}

We will now define the dual space $\mathcal{D}(\Lalg')$ of a general ortholattice $\Lalg'$. It will be the dual Plo\v{s}\v{c}ica space of the lattice-reduct of the ortholattice $\Lalg'$ equipped with the map $g$ as defined above.

\begin{definition}\label{def:dual-orthospace}
The dual of an ortholattice 
$\Lalg'$ is the structure 
$\mathcal{D}(\Lalg')=(P_{\Lalg},E,\T,g)$,
where $g$ is defined as in Lemma~\ref{lem:gmap}
and 
$(P_{\Lalg},E,\T)$
is the dual Plo\v{s}\v{c}ica space of the lattice-reduct $\Lalg$ of $\Lalg'$. 
\end{definition}

The conditions in the proposition below were originally formulated by Marais~\cite{KM20} by translating from Dzik et al.~\cite{DOvA}. They were stated in~\cite{KM20} without proof. 
\begin{theorem}\label{prop:props-of-gmap}
Let $\Lalg$ be an ortholattice. The map $g$ of $\mathcal{D}(\Lalg)$ satisfies the following properties: 
\begin{itemize}
\item[(M1)] $g(g(x))=x$,
\item[(M2)] $xE \subseteq yE \:\Longrightarrow\: Eg(x) \subseteq Eg(y)$,
\item[(M3)] $Ex \subseteq Ey \:\Longrightarrow\: g(x)E \subseteq g(y)E $,
\item[(O)] $\forall x\,\exists y (yE \subseteq xE
\,\&\,
Ey \subseteq Eg(x))$. 
\end{itemize}
\end{theorem}

\begin{proof}
To show (M1), notice that for any $a\in L$ and any MPH $x\in \mathcal{D}(\Lalg)$, $g(g(x))(a)=1$ iff $g(x)(a')=0$ iff $x(a)=1$. Analogously, $g(g(x))(a)=0$ iff $x(a)=0$. Hence $g(g(x))=x$ as required.

To prove (M2), let us assume that $xE \subseteq yE$ and $z\in Eg(x)$ for some $z\in L$, i.e. $y^{-1}(1) \subseteq x^{-1}(1)$ 
by Lemma~\ref{lem:dual} 
and $(z,g(x))\in E$. From the latter it follows that $z^{-1}(1)\cap g(x)^{-1}(0)= \emptyset$. We want to show that   $z\in Eg(y)$, thus $z^{-1}(1)\cap g(y)^{-1}(0)= \emptyset$. Suppose that $a\in z^{-1}(1)\cap g(y)^{-1}(0)$ for some $a\in L$. Then $z(a)=1$ and $g(y)(a)=0$. The latter yields $y(a')=1$, whence $x(a')=1$, thus $g(x)(a)=0$. It follows $a\in z^{-1}(1)\cap g(x)^{-1}(0)$, a contradiction.

One can analogously prove that (M3) holds.

We finally show (O). Let $x\in \mathcal{D}(\Lalg)$ and let  
$\langle F,I \rangle$
be the corresponding MDFIP 
$\langle x^{-1}(1),x^{-1}(0) \rangle$.  
By applying Lemma~\ref{lem:FIo}, we have that $y:= 
\langle I',F' \rangle$ 
is an MDFIP. Let us take firstly the pair $y_0:= 
\langle F,F' \rangle$.  
This is a DFIP since $a\in F\cap F'$ for some $a\in L$ gives $a'\in F'$, whence for the ideal $F'$ we have 
$a\vee a'=1 \in F'$, 
a contradiction. Now let $y_1:=  
\langle F,\overline{F'} \rangle$  
be a DFIP with $\overline{F'}$ 
maximal with respect to being disjoint 
from $F$. Finally, let $y:= 
\langle \overline{F},\overline{F'} \rangle$ 
be a DFIP with $\overline{F}$ 
maximal with respect to being disjoint
from $\overline{F'}$, i.e. $y$ is 
an MDFIP. Now clearly 
$x^{-1}(1)\subseteq y^{-1}(1)$, whence $yE\subseteq xE$.

To prove $Ey\subseteq Eg(x)$, we need to show that $g(x)^{-1}(0)\subseteq y^{-1}(0)$, i.e. that for all $a\in L$, $x(a')=1$ implies $y(a)=0$. Hence we need to show that $a'\in x^{-1}(1)=F$ (which is equivalent to $a=a''\in F'$) implies $a\in y^{-1}(0)=\overline{F'}$. This is evident as $F'\subseteq \overline{F'}$.
\end{proof}

Now we are ready to define an abstract ortho-Plo\v{s}\v{c}ica space.

\begin{definition}
An \emph{ortho-Plo\v{s}\v{c}ica} space is a structure 
$ 
\mathbf{X} = 
(X,E,\tau, g)$ such that $(X,E,\tau)$ is a Plo\v{s}\v{c}ica space and $g : X\to X$ is a continuous map satisfying 
(M1)--(M3) and (O). 
\end{definition}

\section{Dual representation theorems}\label{sec:dual repr}

In Section~\ref{sec:Plospaces} we presented a simplified version of 
the Plo\v{s}\v{c}ica spaces from~\cite{P6} and we then recalled the dual 
representations  
between general lattices with bounds and  
Plo\v{s}\v{c}ica spaces. 
We are now ready to present the dual representation theorems between general ortholattices and the ortho-Plo\v{s}\v{c}ica spaces introduced in the previous section. 

We firstly need to present the dual ortholattice of a general ortho-Plo\v{s}\v{c}ica space.

\begin{definition}
For an ortho-Plo\v{s}\v{c}ica 
space  
$\mathbf{X}=(X,E,\tau,g)$, we define its dual as a lattice with bounds  $\mathcal{E}
(\mathbf{X})=(\mpm{\X}{\twoT_\T}, 
\wedge,\vee,\varphi_0,\varphi_1,\neg)$, 
where $\neg$ assigns to every 
$\varphi\in \mpm{\X}{\twoT_\T}$
a partial map $\neg \varphi: X \to \{0,1\}$ given 
as follows:
$$ (\neg \varphi)(x) = \begin{cases} 1 &\text{ if } \varphi(g(x))=0,\\
0 &\text{ if } \varphi(g(x))=1,\\
- &\text{ otherwise}.\end{cases}$$
\end{definition}

\begin{lemma}\label{lem:neg-is-ortho}
Let $\mathbf{X}=(X,E,\tau,g)$ be an ortho-Plo\v{s}\v{c}ica space. Then $\neg$ as defined above is an orthocomplement on the lattice reduct of $\mathcal{E}(\mathbf{X})$. Hence the dual $\mathcal{E}(\mathbf{X})=(
\mpm{\X}{\twoT_\T},
\wedge,\vee,\varphi_0,\varphi_1,\neg)$ of an ortho-Plo\v{s}\v{c}ica space $\mathbf{X}$  is an ortholattice. 
\end{lemma}

\begin{proof}
We firstly show that $\neg \varphi$ preserves $E$. Suppose that $(x,y)\in E$ for $x,y\in X$ and $(\neg \varphi)(x)=1$ and $(\neg \varphi)(y)=0$. Then $\varphi(g(x))=0$ and $\varphi(g(y))=1$, i.e. $g(x)\in \varphi^{-1}(0)$ and $g(y)\in \varphi^{-1}(1)$. By condition (Ti) 
we obtain from $(x,y)\in E$ that there is $z\in X$ with $zE\subseteq xE$ and $Ez\subseteq Ey$. 
Now the former by (M2) gives $Eg(z) \subseteq Eg(x)$, whence 
$\varphi(g(z))=0$
by Lemma~\ref{lem:E},
while the latter by (M3) gives $g(z)E \subseteq g(y)E$, whence 
$\varphi(g(z))=1$, a contradiction. Hence $\neg \varphi$ preserves $E$.

To show that $\neg \varphi$ is continuous, we notice that $(\neg \varphi)^{-1}(1)=\{x\in X\mid \varphi(g(x))=0\}=\{x\in X\mid g(x)\in \varphi^{-1}(0)\}=g^{-1}(\varphi^{-1}(0))$, 
which is an open set since $\varphi^{-1}(0)$ is open and $g$ is continuous. Analogously, 
$(\neg \varphi)^{-1}(0)=g^{-1}(\varphi^{-1}(0))$ is open.

To show that $\neg \varphi$ is an MPM, it remains to prove the maximality of $\neg \varphi$. Suppose that $\neg \varphi$ can be properly extended to $\psi$ so that there exists $x \in X$ such that $\psi(x)=1$ but $(\neg \varphi)(x)\neq 1$ and $(\neg \varphi)(x)\neq 0$. We will show that $\psi$ cannot be an $E$-preserving extension of $\neg \varphi$. Specifically, we will show that there exists $z$ such that $(\neg \varphi)(z)=0$ and $(x,z) \in E$. 

From $(\neg \varphi)(x) \neq 1$ we get $\varphi(g(x))\neq 0$. Hence there exists $y \in \varphi^{-1}(1)$ such that $(y,g(x))\in E$. By (Ti), there exists $w$ with $wE \subseteq yE$ and $Ew \subseteq Eg(x)$. By Lemma~\ref{lem:E}(i) we get $\varphi(w)=1$. Notice that $\varphi(w)=\varphi(g(g(w)))=1$ and so $(\neg \varphi)(g(w))=0$. By (M3) we get $g(w)E \subseteq xE$. Hence $(x,g(w))\in E$. Now let $g(w)=z$. 

A similar proof can be used for the case of $\psi$ extending $\neg \varphi$ 
such that for some 
$x \notin (\neg \varphi)^{-1}(1)\cup (\neg \varphi)^{-1}(0)$ we have $\psi(x)=0$. 
It will use (Ti) combined with (M2).

Now we show that $\neg \varphi$ 
satisfies 
the orthocomplement identities.

Firstly, it is easy to see that for any $x\in X$, $(\neg \neg \varphi)(x)=1$ iff $(\neg \varphi)(g(x))=0$ iff 
$\varphi(g(g(x)))=1$ 
iff $\varphi(x)=1$. Analogously, $(\neg \neg \varphi)(x)=0$ iff $\varphi(x)=0$. Hence $\neg \neg \varphi = \varphi$ as required.

We will prove that $\neg$ is order-reversing, from which 
the de Morgan laws 
will follow. 
Assume
that for $\varphi,\psi\in 
\mpm{\X}{\twoT_\T}$ 
we have $\varphi 
\leqslant \psi$. 
It follows that $\varphi^{-1}(1)\subseteq \psi^{-1}(1)$, and equivalently, $\psi^{-1}(0) \subseteq \varphi^{-1}(0)$. To show that $\neg \psi 
\leqslant 
\neg \varphi$, we are going to prove  $(\neg \psi)^{-1}(1) \subseteq (\neg \varphi)^{-1}(1)$. So let for $x\in X$, $x\in (\neg \psi)^{-1}(1)$. Then 
$\psi (g(x))=0$, 
which by $\psi^{-1}(0) \subseteq \varphi^{-1}(0)$ gives  
$\varphi (g(x))=0$,  
whence $(\neg \varphi)(x)=1$, i.e. $x\in (\neg \varphi)^{-1}(1)$ as required.

We notice that for all $x\in X$, $(\neg \varphi \wedge \varphi)(x)=0$ iff $x\in (\neg \varphi \wedge \varphi)^{-1}(0)$ iff there is no $y\in (\neg \varphi \wedge \varphi)^{-1}(1)$ with $(y,x)\in E$ iff there is no $y\in (\neg \varphi)^{-1}(1) \cap (\varphi)^{-1}(1)$ with $(y,x)\in E$ iff there is no $y$ such that $\varphi(g(y))=0$ and $\varphi(y)=1$ and $(y,x)\in E$. Suppose now there is $y\in X$ such that $\varphi(g(y))=0$ and $\varphi(y)=1$ and $(y,x)\in E$. By (O), for $y$ there exists $w\in X$ such that $wE\subseteq yE$ and $Ew\subseteq Eg(y)$. As 
$w \in Ew$,
we get $(w,g(y))\in E$ and since $\varphi(g(y))=0$, we obtain $\varphi(w)\ne 1$. However, $wE\subseteq yE$ and $\varphi(y)=1$, whence by Lemma~\ref{lem:E}(i), $\varphi(w)= 1$, a contradiction. Hence $\neg \varphi \wedge \varphi=
\varphi_0$.

(iv) One can analogously prove that $\neg \varphi \vee \varphi=
\varphi_1$.
\end{proof}

We are now ready to present a dual representation theorem for general ortholattices.

\begin{theorem}
For an ortholattice 
$\Lalg'
= (L;\lor ,\land ,',0,1)$, the map $ a\mapsto e_a$ is an ortholattice isomorphism of 
$\Lalg'$ 
onto 
$\mathcal{E}(\mathcal{D}(\Lalg'))=(\mpm{\mathbf{X_L}}{\twoT_\T},
\wedge,\vee,\varphi_0,\varphi_1,\neg)$, 
and hence 
$\Lalg' \cong \mathcal{E}(\mathcal{D}(\Lalg'))$. 
\end{theorem}

\begin{proof}
We know from the Plo\v{s}\v{c}ica representation theorem (see (iii) of Proposition~\ref{lem-eval})   
that it is a lattice isomorphism. We must show that $e_{\neg a} = \neg e_a$ for all $a \in L$. For any $x\in X_L$ we have $(\neg e_a)(x)=1$ iff $e_a(g(x))=0$ iff $g(x)(a)=0$ iff $x(\neg a)=1$ iff $e_{\neg a}(x)=1$. Analogously, 
$(\neg e_a)(x)=0$ iff $e_{\neg a}(x)=0$. 
\end{proof}

Finally, we present a dual representation theorem for 
ortho-Plo\v{s}\v{c}ica spaces.

\begin{theorem}\label{thm:orthospace}
Let $\mathbb{X}=(X,E,\T,g)$ be an ortho-Plo\v{s}\v{c}ica space. Then 
$\mathbb{X} \cong \mathcal{D}(\mathcal{E}(\mathbb{X}))$. 
\end{theorem}

\begin{proof}
Our springboard will be Theorem~\ref{thm:graph-homeomorphic}. 
We define a map $\vartheta : \mathbb{X} \to
\mathcal{D}(\mathcal{L}(\mathbb{X}))$ by
$\vartheta (x)=\varepsilon_x$, where 
 $\varepsilon_x$ is defined as in~Lemma~\ref{lem:defn-v}. 
We showed in~\cite[Theorem 3.10]{P6} that the map $\vartheta$ is well-defined and it is a Plo\v{s}\v{c}ica space isomorphism. We must now show that $\vartheta$ preserves the $g$-map.

Let $\overline{g}$ denote the $g$-map on the space $\mathcal{D}(\mathcal{E}(\mathbb{X}))$. We need to prove that for all $x \in X$, $\vartheta(g(x)) = \overline{g} (\vartheta(x))$. We notice  that for all $\varphi \in 
\mpm{\mathbf{X}}{\twoT_\T}$ 
and $g(x)\in \dom\varphi$, where $x\in X$, we have $\vartheta(g(x))(\varphi) =1$ iff $\varepsilon_{g(x)}(\varphi) =1$ iff $\varphi(g(x))=1$ iff $(\neg \varphi)(x)=0$ iff $\varepsilon_x(\neg \varphi)=0$ iff $\overline{g} (\varepsilon_x)(\varphi)=1$. Analogously, $\varepsilon_{g(x)}(\varphi) =0$ iff $\overline{g} (\varepsilon_x)(\varphi)=0$. Hence for all $x \in X$, 
$$
\vartheta(g(x)) = \varepsilon_{g(x)} = \overline{g} (\varepsilon_x) = \overline{g} 
(\vartheta(x))$$
as required.
\end{proof}

\section{Examples}\label{sec:examples}

In this section, we present the dual ortho-Plo\v{s}\v{c}ica spaces of three different ortholattices: 
the 96-element free orthomodular lattice  $\mathcal{F}_{\mathcal{OM}}(2)$ on $2$ generators, the infinite ortholattice $\mathbf{O}_{\mathbb{Z}}$ (which is not orthomodular), and the infinite modular ortholattice $\mathbf{M}_{\infty}$. 

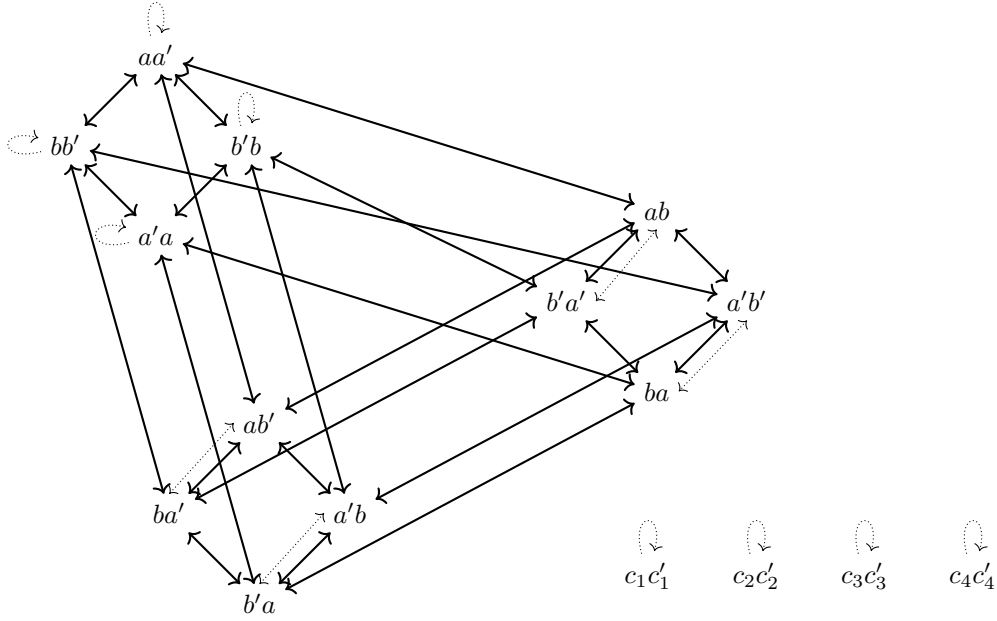
\begin{figure}[ht]
\centering
\begin{tikzpicture}[scale=1.2]

\begin{scope}[xshift=-2cm]
\begin{scope}[xshift=-1.95cm,yshift=2.28cm]
\node[] (aa') at (0,0) {$aa'$};
\node[] (bb') at (-1,-1) {$bb'$};
\node[] (b'b) at (1,-1) {$b'b$};
\node[] (a'a) at (0,-2) {$a'a$};

\path[thick,<->] (aa') edge (bb');
\path[thick,<->] (a'a) edge (bb');
\path[thick,<->] (aa') edge (b'b);
\path[thick,<->] (b'b) edge (a'a);
\end{scope}

\begin{scope}[xshift=3.6cm,yshift=0.55cm]
\node[] (ab) at (0,0) {$ab$}; 
\node[] (a'b') at (1,-1) {$a'b'$}; 
\node[] (ba) at (0,-2) {$ba$}; 
\node[] (b'a') at (-1,-1) {$b'a'$}; 
\path[thick,<->] (ab) edge (a'b');
\path[thick,<->] (a'b') edge (ba);
\path[thick,<->] (ba) edge (b'a');
\path[thick,<->] (b'a') edge (ab);
\end{scope}

\begin{scope}[xshift=-0.8cm,yshift=-1.8cm]
\node[] (ab') at (0,0) {$ab'$}; 
\node[] (a'b) at (1,-1) {$a'b$}; 
\node[] (b'a) at (0,-2) {$b'a$}; 
\node[] (ba') at (-1,-1) {$ba'$}; 
\path[thick,<->] (ab') edge (a'b);
\path[thick,<->] (a'b) edge (b'a);
\path[thick,<->] (b'a) edge (ba');
\path[thick,<->] (ba') edge (ab');
\end{scope}

\path[thick,<->] (aa') edge (ab);
\path[thick,<->] (aa') edge (ab');
\path[thick,<->] (ab') edge (ab);

\path[thick,<->] (a'a) edge (ba);
\path[thick,<->] (b'a) edge (ba);
\path[thick,<->] (a'a) edge (b'a);

\path[thick,<->] (bb') edge (ba');
\path[thick,<->] (bb') edge (a'b');
\path[thick,<->] (ba') edge (b'a');

\path[thick,<->] (b'b) edge (a'b);
\path[thick,<->] (a'b) edge (a'b');
\path[thick,<->] (b'b) edge (b'a');

\path (aa') edge[loop above,densely dotted] (aa');
\path (bb') edge[loop left,densely dotted] (bb');
\path (b'b) edge[loop above,densely dotted] (b'b);
\path (a'a) edge[loop left,densely dotted] (a'a);

\path (ab'.west) edge[<->, densely dotted] (ba'.north);
\path (a'b.west) edge[<->, densely dotted] (b'a.north);

\path (ab.south) edge[<->, densely dotted] (b'a'.east);
\path (a'b'.south) edge[<->, densely dotted] (ba.east);

\end{scope}

\begin{scope}[xshift=1.5cm,yshift=-3.5cm]
\node[] (c1c1') at (0,0) {$c_1c_1'$};
\path (c1c1') edge[loop above,densely dotted] (c1c1');
\node[] (c2c2') at (1.2,0) {$c_2c_2'$};
\path (c2c2') edge[loop above,densely dotted] (c3c3');
\node[] (c3c3') at (2.4,0) {$c_3c_3'$};
\path (c3c3') edge[loop above,densely dotted] (c3c3');
\node[] (c4c4') at (3.6,0) {$c_4c_4'$};
\path (c4c4') edge[loop above,densely dotted] (c4c4');
\end{scope}

\end{tikzpicture}
\caption{The dual digraph of $\mathcal{F}_{\mathcal{OM}}(2)$.  The $g$-map is indicated with the dotted arrows.} \label{fig:MO2}
\end{figure}

\begin{example}
The free orthomodular lattice  $\mathcal{F}_{\mathcal{OM}}(2)$ on $2$ generators 
is isomorphic to 
$ \mathsf{MO}_2 \times 
\mathbf{2}^4$ 
(cf. ~\cite{B84} or~\cite{HKPW97}). 
Let $a$, $b$, $a'$ and $b'$ be the atoms of $\mathsf{MO}_2$ 
and let $\bot$, $\top$ be the bounds of $\mathsf{MO}_2$.
For $\mathbf{2}^4$, let 
$\underline{0}$ 
 be the bottom element, $\underline{1}$
the top element,
and let $c_1$, $c_2$, $c_3$, $c_4$ be the atoms. 
There are only two types of  MDFIPs 
of $\mathcal{F}_{\mathcal{OM}}(2)$  
(equivalently, MPHs from 
$\mathcal{F}_{\mathcal{OM}}(2)$ to $\mathbf{2}$). The first type is 
$\langle {\uparrow}(x,\underline{0}), {\downarrow}(y,\underline{1}) \rangle$ where $x,y$ are atoms of $\mathsf{MO}_2$. The second type is 
$\langle {\uparrow}(\bot,c_i),{\downarrow}
(\top,c'_i)
\rangle$. 
For simplicity, in our figure we write $xy$ for an MDFIP of the first form, and $c_ic_i'$ for an MDFIP of the second form. 
The dual digraph of $\mathcal{F}_{\mathcal{OM}}(2)$ can be seen in Figure~\ref{fig:MO2}.  
It is 
the disjoint union of the dual digraph 
of $\mathsf{MO}_2$ 
and four isolated vertices with 
loops representing the dual of $\mathbf{2}^4$, so the 96-element algebra is represented by the 16-element dual digraph.  

The elements of the algebra are the MPMs from the 16-element dual digraph in Figure~\ref{fig:MO2} into the two-element digraph $\twoT=(\{0,1\}, \le)$. These MPMs can freely map the four isolated vertices into $0$ and $1$, which clearly corresponds to the factor $\mathbf{2}^4$ of $\mathcal{F}_{\mathcal{OM}}(2)$. The restrictions of the MPMs to the dual digraph of $\mathsf{MO}_2$  
obviously correspond to the factor $\mathsf{MO}_2$ of $\mathcal{F}_{\mathcal{OM}}(2)$.

Since the algebra (and hence the dual space) is finite, it is not difficult to show that the topology on the dual space must be the discrete topology. 

The author of~\cite{B84} made considerable effort to present a diagram of the 96-element algebra $\mathcal{F_{OM}}$, see~\cite[Figure 18]{B84}. 
The 16-element digraph in Figure~\ref{fig:MO2} emphasizes the efficiency of the dual representation of the same algebra. 
\end{example}

\begin{example}\label{ex:OZ}

Consider the infinite ortholattice $\mathbf{O}_{\mathbb{Z}}$ in Figure~\ref{fig:OZ}. The orthocomplement is defined by $(a_j)'=b_{-j}$ 
($j\in \mathbb{Z}$). 
Consider the filters $F_a=\{a_i \mid i \in \mathbb{Z}\}\cup\{1\}$ and $F_b=\{b_j\mid j \in \mathbb{Z}\}\cup\{1\}$, and the ideals 
$I_a=\{a_k \mid k \in \mathbb{Z}\}\cup\{0\}$ and 
$I_b=\{b_{\ell} \mid {\ell} \in \mathbb{Z}\}\cup\{0\}$.

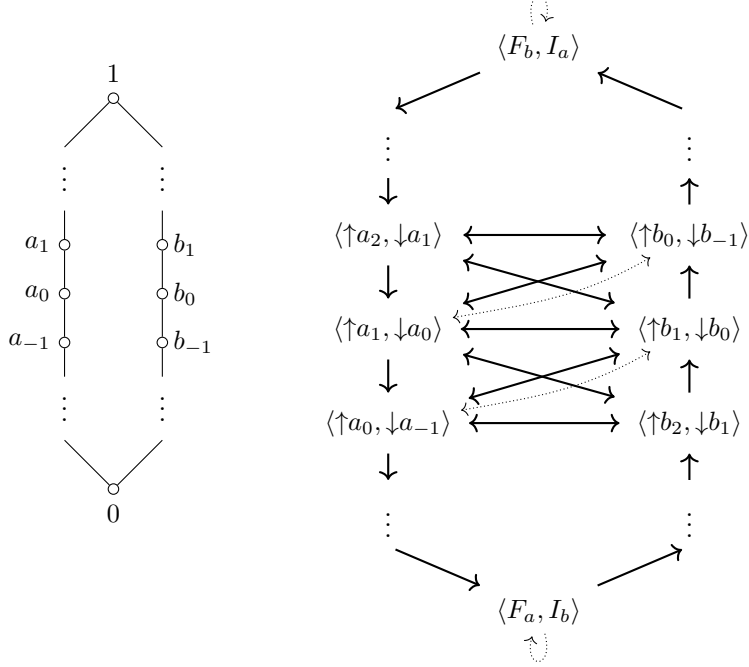
\begin{figure}[ht]
\centering
\begin{tikzpicture}[scale=0.65]
\begin{scope}[yshift=6.5cm]
\node[unshaded] (bot) at (0,-4) {};
\node[unshaded] (a1) at (-1,1) {};
\node[unshaded] (a0) at (-1,0) {};
\node[unshaded] (a-1) at (-1,-1) {};
\node[unshaded] (b1) at (1,1) {};
\node[unshaded] (b0) at (1,0) {};
\node[unshaded] (b-1) at (1,-1) {};
\node[unshaded] (top) at (0,4) {};
\node[] (rght) at (1,2.5) {$\vdots$};
\node[] (lft) at (-1,2.5) {$\vdots$};
\node[] (rght-bottom) at (1,-2.2) {$\vdots$};
\node[] (lft-bottom) at (-1,-2.2) {$\vdots$};

\draw[order] (-1,-3)--(bot)--(1,-3) ;
\draw[order] (-1,3)--(top)--(1,3) ;
\draw[order] (-1,-1.7)--(a-1)--(a0)--(a1)--(-1,1.7); 
\draw[order] (1,-1.7)--(b-1)--(b0)--(b1)--(1,1.7); 
\node[label,anchor=north] at (bot) {$0$};
\node[label,anchor=east,xshift=1pt] at (a-1) {$a_{-1}$};
\node[label,anchor=east,xshift=1pt] at (a0) {$a_0$};
\node[label,anchor=east,xshift=1pt] at (a1) {$a_1$};
\node[label,anchor=west,xshift=-2pt] at (b-1){$b_{-1}$};
\node[label,anchor=west,xshift=-2pt] at (b0){$b_0$};
\node[label,anchor=west,xshift=-2pt] at (b1){$b_1$};
\node[label,anchor=south] at (top) {$1$};
\end{scope}

\begin{scope}[xshift=8.7cm, scale=0.77]

\node[] (dots-LB) at (-4, 2.5) {$\vdots$};
\node[] (a0a-1) at (-4, 5) {$\MDFIP{a_0}{a_{-1}}$};
\node[] (a1a0) at (-4, 7.5) {$\MDFIP{a_1}{a_0}$};
\node[] (a2a1) at (-4, 10) {$\MDFIP{a_2}{a_1}$};
\node[] (dots-LT) at (-4, 12.5) {$\vdots$};

\node[] (dots-RT) at (4, 12.5) {$\vdots$};
\node[] (b0b-1) at (4, 10) {$\MDFIP{b_0}{b_{-1}}$};
\node[] (b1b0) at (4, 7.5) {$\MDFIP{b_1}{b_0}$};
\node[] (b2b1) at (4, 5) {$\MDFIP{b_2}{b_1}$};
\node[] (dots-RB) at (4, 2.5) {$\vdots$};

\node[] (FbIa) at (0, 15) {$\langle F_b,I_a \rangle$};
\node[] (FaIb) at (0, 0) {$\langle F_a, I_b \rangle$};

\path[thick,->,shorten >=3pt,shorten <=3pt] (dots-LB.south) edge (FaIb.north west);
\path[thick,->,shorten >=3pt,shorten <=3pt] (FaIb.north east) edge (dots-RB.south);

\path[thick,->,shorten >=3pt,shorten <=3pt] (dots-RT.north) edge (FbIa.south east);
\path[thick,->,shorten >=3pt,shorten <=3pt] (FbIa.south west) edge (dots-LT.north);

\path[thick,->,shorten >=3pt,shorten <=3pt] (dots-LT.south) edge  (a2a1.north);
\path[thick,->,shorten >=3pt,shorten <=3pt] (a2a1.south) edge  (a1a0.north);
\path[thick,->,shorten >=3pt,shorten <=3pt] (a1a0.south) edge  (a0a-1.north);
\path[thick,->,shorten >=3pt,shorten <=3pt] (a0a-1.south) edge  (dots-LB.north);

\path[thick,->,shorten >=3pt,shorten <=3pt] (b2b1.north) edge  (b1b0.south);
\path[thick,->,shorten >=3pt,shorten <=3pt] (b1b0.north) edge  (b0b-1.south);
\path[thick,->,shorten >=3pt,shorten <=3pt] (b0b-1.north) edge  (dots-RT.south);
\path[thick,->,shorten >=3pt,shorten <=3pt] (dots-RB.north) edge  (b2b1.south);

\path[thick,<->,shorten >=3pt,shorten <=3pt] (a0a-1.east) edge  (b2b1.west);
\path[thick,<->,shorten >=3pt,shorten <=3pt] (a1a0.east) edge (b1b0.west);
\path[thick,<->,shorten >=3pt,shorten <=3pt] (a0a-1.north east) edge  (b1b0.south west);
\path[thick,<->,shorten >=5pt,shorten <=4pt] (a1a0.south east) edge  (b2b1.north west);

\path[thick,<->,shorten >=5pt,shorten <=4pt] (a1a0.north east) edge  (b0b-1.south west);
\path[thick,<->,shorten >=5pt,shorten <=4pt] (a2a1.east) edge  (b0b-1.west);
\path[thick,<->,shorten >=5pt,shorten <=4pt] (a2a1.south east) edge  (b1b0.north west);

\path (FbIa) edge[loop above,densely dotted] (FbIA);
\path (FaIb) edge[loop below,densely dotted] (FaIb);
\path[<->] (a1a0) edge[out=10,in=210,densely dotted] (b0b-1); 
\path[<->] (a0a-1) edge[out=10,in=210,densely dotted] (b1b0); 
\end{scope}

\end{tikzpicture}
\caption{The infinite ortholattice  $\mathbf{O}_{\mathbb{Z}}$ and its dual digraph. The involutive map $g$ is indicated with dotted arrows. Not all edges of the digraph are included. Refer to Table~\ref{table:digraph-of-OZ} for a complete description.}
\label{fig:OZ}
\end{figure}

The topology on the dual digraph can be described by subbasic closed sets of the form: 
$$V_c=\{\,\langle F,I \rangle \mid 
c\in I 
\,\}\quad \text{and} \quad 
W_c 
=\{\,\langle F,I\rangle\mid 
c \in F 
\,\}.$$
For $c=a_k$ 
($k\in \mathbb{Z}$) 
 we get:
$$V_c= \{\MDFIP{a_{j+1}}{a_j} \mid j \geqslant k \} \cup\{\langle F_b,I_a\rangle\}\text{ and }  W_c = 
\{ \MDFIP{a_{j+1}}{a_j} \mid j<k \} \cup \{\langle F_a, I_b\rangle\}.
$$
For $c=b_k$ ($k\in \mathbb{Z}$) 
 we get:
$$V_c= \{\MDFIP{b_{j+1}}{b_j} \mid j \geqslant k \} \cup\{\langle F_a,I_b\rangle\}\text{ and }  W_c = 
\{ \MDFIP{b_{j+1}}{b_j} \mid j<k \} \cup \{\langle F_b, I_a\rangle\}.
$$

Now for $k \in \mathbb{Z}$ we can calculate 
$W_{a_{k+1}}\cap V_{a_k}=
\{\langle 
{\uparrow}a_{k+1},{\downarrow}a_k
\rangle\}$, $W_{b_{k+1}}\cap V_{b_k} = \{ \langle 
{\uparrow}b_{k+1},{\downarrow}b_k
\rangle\}$, 
$\bigcap \{\, W_{a_i} \mid i \in \mathbb{Z}\,\} = \{ \langle 
F_a, I_b\rangle\}$,
and  $\bigcap \{\, V_{b_i}\mid i\in\mathbb{Z}\,\} = \{ \langle 
F_a,I_b\rangle\}$.
From this we see that every singleton from the digraph is closed, and hence  the topology is Hausdorff. 

The edges of the digraph are given in Table~\ref{table:digraph-of-OZ}. In the row of $\langle F,I\rangle$ and the column of $\langle G,J \rangle$ we write 1 iff $\langle F,I \rangle E \langle G,J \rangle$ iff $F \cap J = \emptyset$. 

 \begin{table}[h!]
\centering 
\begin{tabular}{|  c ||  c  |  c   | c  | c|c|c|  }
\hline & $\langle F_a,I_b \rangle$ & $\langle F_b, I_a\rangle $ & $\MDFIP{a_{j+1}}{a_j}$ & $\MDFIP{a_{k+1}}{a_k}$ &$\MDFIP{b_{j+1}}{b_j}$  &$\MDFIP{b_{k+1}}{b_k}$   \\\hline
\hline $\langle F_a, I_b\rangle $& 1 & 0 & 0& 0 &1 &1  \\
\hline $\langle F_b, I_a\rangle$ & 0 & 1 &1 &1 &0 &0  \\
\hline $\MDFIP{a_{j+1}}{a_j}$ &1 &0 &1 &0 & 1& 1\\
\hline $\MDFIP{a_{k+1}}{a_k}$ &1 &0 &1 &1 & 1&1\\  
\hline $\MDFIP{b_{j+1}}{b_j}$ &0 &1 &1 &1 &1 &0\\  
\hline $\MDFIP{b_{k+1}}{b_k}$ &0 &1 &1 &1 &1 &1\\ \hline 
\end{tabular}

\vspace{4mm}

\caption{Table of edges for the dual digraph of 
$\mathbf{O}_{\mathbb{Z}}$ where 
$j,k \in \mathbb{Z}$
and $j<k$.}
\label{table:digraph-of-OZ}
\end{table} 

Table~\ref{table:MPHsandMPMs-OZ} indicates how there are essentially four types of MPHs.
The rows of the table represent the MPHs while the columns of the table represent the MPMs. More precisely, 
each row shows the 
images of elements  
of $\mathbf{O}_{\mathbb{Z}}$ under the MPH corresponding to the presented MDFIP. The algebra elements are then recovered via the MPMs whose 
images of the digraph vertices 
can be read off columnwise. The join of two MPMs is the column with at least as many 1s as the union of the two columns' 1s.

 \begin{table}[ht]
\centering 
\begin{tabular}{|  c ||  c  |  c   | c  | c|c|c|  }
\hline & $0$ & $a_i$ &$a_\ell$  &$b_i$  &$b_\ell$  & $1$   \\\hline
\hline $\langle F_a, I_b\rangle $& 0 &1  &1 &0  &0 &1  \\
\hline $\langle F_b, I_a\rangle$ & 0 & 0 &0 &1 &1 &1  \\
\hline $\MDFIP{a_{j+1}}{a_j}$ &0 &1 &0 &$-$ & $-$ & 1\\
\hline $\MDFIP{b_{j+1}}{b_j}$ &0 & $-$ & $-$ & 1 &0 &1\\  
\hline 
\end{tabular}

\vspace{4mm}
\caption{MPHs and MPMs for  $\mathbf{O}_{\mathbb{Z}}$ with 
$i,j,\ell \in \mathbb{Z}$,
$i>j$ and $\ell \leqslant j$.}
\label{table:MPHsandMPMs-OZ}
\end{table} 
\end{example}

\begin{example}\label{ex:Minfty}
Consider the infinite orthomodular lattice of height two  $\mathbf{M}_\infty$ (Figure~\ref{fig:Minfty}). The atoms are labelled by elements of $\mathbb{Z}{\setminus}\{0\}$ and define the orthocomplement by $\bot'=\top$ and $z'=-z$. 
The elements of the dual space are MDFIPs of the form $\MDFIP{a}{b}$ for $a,b \in \mathbb{Z}{\setminus}\{0\}$ and $a \neq b$.

\begin{figure}[ht]
\centering
\begin{tikzpicture}[scale=0.8]
\begin{scope}[yshift=2.5cm]
\node[unshaded] (bot) at (0,-2) {};
\node[unshaded] (a-3) at (-2.5,0) {};
\node[unshaded] (a-2) at (-1.5,0) {};
\node[unshaded] (a-1) at (-0.5,0) {};
\node[unshaded] (a1) at (0.5,0) {};
\node[unshaded] (a2) at (1.5,0) {};
\node[unshaded] (a3) at (2.5,0) {};
\node[unshaded] (top) at (0,2) {};
\node[] (dots-r) at (3.3,0) {$\hdots$};
\node[] (dots-l) at (-3.3,0) {$\hdots$};
\node[] (dots-rr) at (3.9,0) {$\hdots$};
\node[] (dots-ll) at (-3.9,0) {$\hdots$};
\node[] (dots-rrr) at (4.5,0) {$\hdots$};
\node[] (dots-lll) at (-4.5,0) {$\hdots$};

\draw[order] (bot)--(a-3)--(top);
\draw[order] (bot)--(a-2)--(top);
\draw[order] (bot)--(a-1)--(top);
\draw[order] (bot)--(a1)--(top);
\draw[order] (bot)--(a2)--(top);
\draw[order] (bot)--(a3)--(top);

\draw[order] (bot)--(-4,-0.3);
\draw[order] (-4,0.3)--(top);
\draw[order] (bot)--(4,-0.3);
\draw[order] (4,0.3)--(top);

\node[label,anchor=west,xshift=-3pt,yshift=-6pt] at (a1) {$1$};
\node[label,anchor=west,xshift=-3pt,yshift=-6pt] at (a2) {$2$};
\node[label,anchor=west,xshift=-3pt,yshift=-6pt] at (a3) {$3$};

\node[label,anchor=east,xshift=5pt,yshift=-6pt] at (a-1) {$-1$};
\node[label,anchor=east,xshift=5pt,yshift=-6pt] at (a-2) {$-2$};
\node[label,anchor=east,xshift=5pt,yshift=-6pt] at (a-3) {$-3$};

\node[label,anchor=north] at (bot) {$\bot$};
\node[label,anchor=south] at (top) {$\top$};
\end{scope}

\end{tikzpicture}
\caption{The infinite orthomodular lattice $\mathbf{M}_{\infty}$.}
\label{fig:Minfty}
\end{figure}
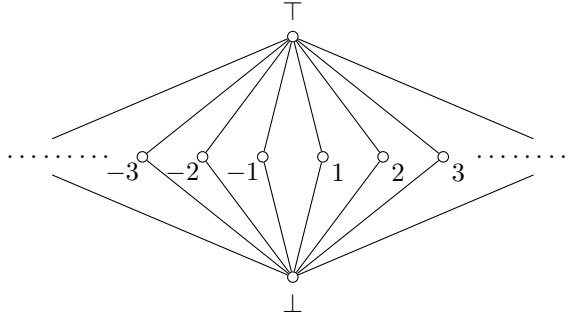

 \begin{table}[ht]
\centering 
\begin{tabular}{|  c ||  c  |  c   | c  | c|c|  }
\hline & $\bot$ & $j$ &$k$  &$\ell$  &   $\top$   \\\hline
\hline $\MDFIP{j}{k}$ & 0 &1  &0 &$-$  &1   \\ 
\hline $\MDFIP{j}{\ell}$ & 0 &1  &$-$ &0  &1   \\ 
\hline $\MDFIP{k}{\ell}$ & 0 &$-$  &1 &0  &1   \\ 
\hline $\MDFIP{k}{j}$ & 0 &0  &1 &$-$ &1   \\ 
\hline $\MDFIP{\ell}{j}$ & 0 &0  &$-$ &1  &1   \\ 
\hline $\MDFIP{\ell}{k}$ & 0 &$-$  &0 &1  &1   \\ 
\hline 
\end{tabular}

\vspace{4mm}
\caption{MPHs and MPMs for  $\mathbf{M}_{\infty}$. We have $j,k,\ell \in \mathbb{Z}{\setminus}\{0\}$ with all of them distinct. }
\label{table:MPHsandMPMs-Minfty}
\end{table} 

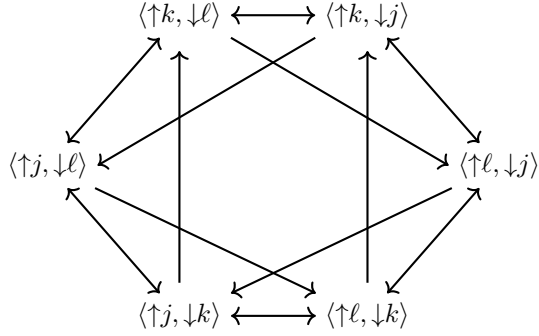
\begin{figure}[ht]
\centering
\begin{tikzpicture}[scale=1]
\begin{scope}[xshift=4.25cm]
\node[] (jk) at (-1.25,0) {$\MDFIP{j}{k}$};
\node[] (lk) at (1.25,0) {$\MDFIP{\ell}{k}$};
\node[] (jl) at (-3,2) {$\MDFIP{j}{\ell}$};
\node[] (lj) at (3,2) {$\MDFIP{\ell}{j}$};
\node[] (kl) at (-1.25,4) {$\MDFIP{k}{\ell}$};
\node[] (kj) at (1.25,4) {$\MDFIP{k}{j}$};

\path[thick,<->] (kl) edge (kj);
\path[thick,<->] (kj) edge (lj);
\path[thick,<->] (lj) edge (lk);
\path[thick,<->] (lk) edge (jk);
\path[thick,<->] (jk) edge (jl);
\path[thick,<->] (jl) edge (kl);

\path[thick,->] (kl.south east)  edge  (lj.west);
\path[thick,->] (lj) edge (jk.north east);
\path[thick,->,shorten >=5pt,shorten <=4pt] (jk.north) edge (kl.south);

\path[thick,->,] (kj.south west)  edge  (jl.east);
\path[thick,->] (jl)  edge  (lk.north west);
\path[thick,->,shorten >=5pt,shorten <=4pt] (lk.north) edge  (kj.south);

\end{scope}

\end{tikzpicture}
\caption{A part 
of the dual digraph of $\mathbf{M}_\infty$ (six vertices with distinct $j,k,l \in \mathbb{Z}{\setminus}\{0\}$).}
\label{fig:Minfty-dual-local}
\end{figure}

Table~\ref{table:MPHsandMPMs-Minfty} indicates 
MPHs and MPMs of $\mathbf{M}_\infty$ for distinct $j,k,l \in \mathbb{Z}{\setminus}\{0\}$. Each row shows the 
images of elements 
of $\mathbf{M}_{\infty}$ under the MPH corresponding to the presented MDFIP. The ortholattice elements are again recovered via the MPMs whose 
images of the digraph vertices   
can be read off columnwise and the join of two MPMs is the column with at least as many 1s as the union of the two columns' 1s.  

We observe  that $\MDFIP{a}{b} E \MDFIP{c}{d}$ iff $a \neq d$. In Figure~\ref{fig:Minfty-dual-local} we see a part of the dual digraph of $\mathbf{M}_\infty$ depicting six vertices corresponding to distinct $j,k,l \in \mathbb{Z}{\setminus}\{0\}$.

We note 
that for any $j,k \in \mathbb{Z}$ with $j\neq k$ we have $W_j \cap V_k = \{\langle {\uparrow}j, {\downarrow}k\rangle\}$ and hence the topology on the dual digraph of $\mathbf{M}_\infty$ is 
Hausdorff. 
\end{example}

\subsection*{Acknowledgements}
The first author acknowledges the hospitality of Matej Bel University during
a visit in August-September 2024. 
The second author acknowledges his appointment as a Visiting Professor at the University
of Johannesburg between 2020-2026 and also acknowledges the hospitality of the University
of Johannesburg during a two-week visit in July-August 2025.

The authors would also like to thank to Professor Gerhard D\"orfer for the hospitality of the Vienna University of Technology during a one-week visit in August 2022, where this study of the duals of ortholattices started. During this short visit as well as during many other visits and meetings the authors also greatly enjoyed the company of Professor G\"unther Eigenthaler, who passed away on February 14, 2025, just five days after reaching his jubilee of 75 years. He will always be remembered in our hearts and we are privileged to dedicate this paper to his memory.



\begin{thebibliography}{9}

\bibitem{B84} 
\uppercase{Beran, L.}:
\textit{Orthomodular lattices. Algebraic Approach},  
Academia, Prague, 1984.

\bibitem{BN36} 
\uppercase{Birkhoff, G.---von Neumann, J.}:
\textit{The logic of quantum mechanics}, Ann. of Math. \textbf{37} (1936), 823--843. Also in: \textit{J. von Neumann Collected Works}, Vol IV, Pergamon Press, Oxford, 1961, 105--125.

\bibitem{ACthesis} \uppercase{Craig, A.P.K.}: \textit{Canonical extensions of bounded lattices and natural duality for default bilattices}, DPhil Thesis, University of Oxford, 2012.  

\bibitem{P3}
\uppercase{Craig, A.P.K.---Gouveia M.J.---Haviar, M.}:
\textit{TiRS graphs and TiRS frames: a new setting for duals of canonical extensions}, Algebra Universalis \textbf{74} (2015), 123--138.

\bibitem{P4} 
\uppercase{Craig, A.P.K.---Gouveia M.J.---Haviar, M.}: 
\textit{Canonical extensions of lattices are more than perfect}, In memory of Bj\'arni J\'onsson, Algebra Universalis \textbf{83:12}, 2022.  

\bibitem{P2}
\uppercase{Craig, A.P.K.---Haviar, M.}:
\textit{Reconciliation of approaches to the construction of canonical extensions of bounded lattices}, 
Math. Slovaca \textbf{64} (2014), 1335--1356.

\bibitem{P7} 
\uppercase{Craig, A.P.K.---Haviar, M.---Marais, K.}: 
\textit{Dual digraphs of finite meet-distributive and modular lattices}, CUBO \textbf{26} (2024), 279--302. 

\bibitem{P1}
\uppercase{Craig, A.P.K.---Haviar, M.---Priestley, H.A.}:
\textit{A fresh perspective on the canonical extensions of bounded lattices}, Appl. Categ. Struct. \textbf{21} (2013), 725--749.

\bibitem{P5} 
\uppercase{Craig, A.P.K.---Haviar, M.---S\~{a}o Jo\~{a}o, J.}: 
\textit{Dual digraphs of finite semidistributive lattices}, CUBO \textbf{24} (2022), 369--392.

\bibitem{P6} 
\uppercase{Craig, A.P.K.---Haviar, M.--- S\~{a}o Jo\~{a}o, J.}: 
\textit{Dual spaces of lattices and semidistributive lattices}, submitted. Available at: \url{https://arxiv.org/abs/2509.04417}.


\bibitem{D05} 
\uppercase{D\"orfer, G.}:
\textit{Noncommutative symmetric differences in orthomodular lattices}, Internat. J. Theoret. Phys.  \textbf{44} (2005), 885--896.

\bibitem{DDL96} 
\uppercase{D\"orfer, G.---Dvure\v censkij, A.---L\"anger, H.}:
\textit{Symmetric difference in orthomodular lattices}, 
Math. Slovaca \textbf{46} (1996), 435--444.

\bibitem{D86}
\uppercase{Dvure\v{c}enskij, A.}:
\textit{On two problems of quantum logics}, Math. Slovaca \textbf{36} (1986), 253--265.

\bibitem{DP00}
\uppercase{Dvure\v{c}enskij, A.---Pulmannov\'{a}, S.}:
\textit{New Trends in Quantum Structures},
Kluwer Acad. Publ./Ister Sci., Dordrecht/Bratislava, 2000.

\bibitem{DOvA}
\uppercase{Dzik, W.---Orlowska, E.---van Alten, C.}: 
\textit{Relational Representation Theorems
for General Lattices with Negations}, {I}n: Lecture Notes in Computer Science, 4136,
Berlin, pp. 162–176, 2006.

\bibitem{GW99}
\uppercase{Ganter, B.---Wille R.}: 
\textit{Formal Concept Analysis: Mathematical Foundations}, Springer, 1999. 


\bibitem{HKPW97}
\uppercase{Haviar, M.---Kon\^opka, P.---Priestley, H.A.---Wegener, C.B.}: 
\textit{Finitely generated free modular ortholattices I}, International Journal of Theoretical Physics \textbf{36} (1997), 2639--2660.


\bibitem{H37}
\uppercase{Husimi, K.}:
\textit{Studies on the foundations of the quantum mechanics I.}, Proc. Phys.-Math. Soc. Japan \textbf{19} (1937), 766--789.


\bibitem{K83}
\uppercase{Kalmbach, G.}:
\textit{Orthomodular lattices}, London Mathematical Society Monographs \textbf{18}, Academic Press, Inc. [Harcourt Brace Jovanovich, Publishers], London-New York, 1983.

\bibitem{KM20} 
\uppercase{Marais, K.}: 
\textit{Dual structures of ortholattices}, Honours Thesis, University of Johannesburg, 2020. 


\bibitem{Plos95}
\uppercase{Plo\v{s}\v{c}ica, M.}:
\textit{A natural representation of bounded lattices}, Tatra Mountains Math. Publ. \textbf{5} (1995), 75--88.

\bibitem{Pr70}
\uppercase{Priestley, H.A.}: 
\textit{Representation of distributive lattices by means of ordered Stone spaces}, Bull. London Math. Soc. \textbf{2} (1970), 186--190.

\bibitem{Pr72}
\uppercase{Priestley, H.A.}: 
\textit{Ordered topological spaces and the representation of distributive lattices}, Proc. London Math. Soc. \textbf{24} (1972), 507--530.

\bibitem{PP91}
\uppercase{Pt\'ak, P. ---Pulmannov\'a, S.}:
\textit{Orthomodular structures as quantum logics}, Kluwer Academic Publishers, Dordrecht-Boston-London, 1991.

\bibitem{U78}
\uppercase{Urquhart, A.}: 
\textit{A topological representation theory for lattices}, Algebra Universalis \textbf{8} (1978), 45--58.

 
\end{thebibliography}
\end{document}